\documentclass{birkjour}

\usepackage[T1]{fontenc}
\usepackage{graphicx}
\usepackage{cite,enumerate}

\usepackage{color}
\definecolor{MyLinkColor}{rgb}{0,0,0.4}
\renewcommand{\bfseries}{\color[rgb]{0.3,0.2,0} \normalfont}

\newcommand{\im}{\mathop{\rm Im}\nolimits}
\newcommand{\id}{\mathop{\rm id}\nolimits}

\newcommand{\p}{\partial}

\newcommand{\Ker}{\mathop{\rm Ker}\nolimits}
\newcommand{\f}{\varphi}
\newcommand{\e}{\varepsilon}

\newcommand{\h}{\rho}

\newcommand{\ov}{\overline}
\newcommand{\wh}{\widehat}

\newcommand{\0}{\Omega}

\newcommand{\V}{\mathcal{V}}

\newcommand{\s}{\mathbb{S}}

\newcommand{\R}{\mathbb{R}}

\newcommand{\C}{\mathbb{C}}
\newcommand{\x}{\mathcal{T}}
\newcommand{\mS}{\mathcal{S}}
\newcommand{\N}{\mathbb{N}}

\newcommand{\Z}{\mathbb{Z}}

\newtheorem{thm}{Theorem}[section]

\newtheorem{lem}[thm]{Lemma}

\theoremstyle{remark}

\numberwithin{equation}{section}
\usepackage[colorlinks=true,linkcolor=MyLinkColor,citecolor=MyLinkColor]{hyperref}

\begin{document}

\title[Bifurcation analysis]{Bifurcation analysis for a free boundary problem modeling tumor growth}

\author[J. Escher]{Joachim Escher}
\address{Institut f{\"u}r Angewandte Mathematik, Leibniz Universit{\"a}t Hannover, Welfengarten~1, 30167 Hannover, Germany. }
\email{escher@ifam.uni-hannover.de}

\author[A.-V. Matioc]{Anca-Voichita Matioc}
\address{Institut f{\"u}r Angewandte Mathematik, Leibniz Universit{\"a}t Hannover, Welfengarten~1, 30167 Hannover, Germany. }
\email{matioca@ifam.uni-hannover.de}



\begin{abstract}
In this paper we deal with a free boundary problem modeling the growth of nonnecrotic tumors.
The tumor is treated as an incompressible fluid, the tissue elasticity is neglected and no chemical inhibitor species 
are present. 
We re-express the mathematical model as an operator equation and  by
using a bifurcation argument we prove that there exist stationary solutions of the problem which are not radially symmetric.
\end{abstract}

\maketitle

\section{Introduction and the main result}
Cristini et all. obtained in \cite{VCris}  
a new mathematical formulation of an existing model (see \cite{BC, FR, Gr}) which describes the  evolution of  nonnecrotic tumors in both vascular and avascular regimes.
As widely used in the modelling, the  tumor is treated as an incompressible fluid and tissue elasticity is neglected. 
Cell-to-cell adhesive forces are modeled by surface tension at the tumor-tissue interface. 
The growth of the tumor is governed by a balance between cell-mitosis and apoptosis (programed cell-death). 
The rate of mitosis  depends on the concentration of nutrient and no inhibitor chemical species are present. 
This new model is obtained by considering different intrinsic time and length scales for the tumor evolution which are integrated  by means of algebraic manipulations into the model.
The model presented in \cite{BC, FR, Gr}, has been studied extensively by different authors  \cite{BF05, Cui, CE, CE1, FR}.  
It is known that the  moving boundary problems associated to it  are well-posed locally in time  \cite{CE1, FR} and, as a further common characteristic, 
there exists, for parameters in a certain range, a unique radially symmetric equilibrium 
\cite{Cui, CE, CE1, FR01}.
This results have been verified to hold true also for the model deduced in \cite{VCris}, cf. \cite{VCris, EM, EM1}.
The authors of \cite{CE, CE1, CEZ, FR01, ZC} show by using the  theorem on bifurcation from simple eigenvalues due to Crandall and Rabinowitz
 that in the situations they consider there exists besides the unique radially symmetric equilibrium other
nontrivial equilibria.
Though the problems they consider are different, these nontrivial steady-state solutions are asymptotically identical
near the circular equilibrium.
Numerical experiments  suggest also for the model 
\cite{VCris}   that there may exist  stationary solutions which are no longer radially symmetric.

In this paper we focus on the general, i.e. non-symmetric situation, when the tumor domain is arbitrary and look for nontrivial steady-states of the model \cite{VCris}.
Additionally to the well-posedness of the associated moving boundary problem, stability properties of the unique radially symmetric solution  are  established in \cite{EM1}.
Particularly, it is shown that if $G$, the  rate of mitosis relative to the relaxation mechanism is large, then the circular equilibrium is unstable, which also suggests existence of 
nontrivial stationary solutions.
When studying the set of equilibria we deal with a free boundary problem which is reduced to an operator equation between certain subspaces of the small H\"older spaces  over the unit circle $h^{m+\beta}(\s)$. 
We apply then the theorem on bifurcations from simple eigenvalues to this equation and obtain infinitely many bifurcation branches consisting only of stationary solutions of our model.
Near the circular equilibrium these solutions match perfectly the ones found in \cite{CE, CE1, CEZ, FR01}.

The outline of the paper is as follows: we present in the subsequent section the mathematical model and the main result, Theorem \ref{BifTh}.
Section 3 is dedicated to the proof of the  Theorem \ref{BifTh}.

\section{The mathematical model and the main result}

The two-dimensional system associated to the model \cite{VCris} is described in detail in \cite{EM}.
The steady-state solution of the  moving boundary problem presented there are precisely the solutions of the free boundary problem 
\begin{equation}\label{eq:problem}
\left \{
\begin{array}{rllllll}
\Delta \psi &=& f(\psi )  &\text{in} & \Omega, \\[1ex]
\Delta p &=& 0 & \text{in}& \Omega,  \\[1ex]
\psi &=& 1 & \text{on}& \partial \Omega,  \\[1ex]
p&=& \kappa_{\p\0}- AG \displaystyle\frac{ |x|^2}{4} &\text{on}& \partial\Omega,\\[1ex]
G\displaystyle\frac{\p \psi}{\p n} -\displaystyle\frac{\p p}{\p n} -AG \displaystyle\frac{n\cdot x}{2} &=&0
 & \text{on}& \partial \Omega.  \\[1ex]
\end{array}
\right.
\end{equation}
The  fully nonlinear system  \eqref{eq:problem} consists of two decoupled Dirichlet problems, one 
for the rate $\psi$ at which nutrient is added to the tumor domain $\0$, and  one for the pressure $p$ inside the tumor.
These  two variables are coupled by the fifth equation of \eqref{eq:problem}.
Hereby $\kappa_{\p\0}$  stands for the curvature of $\p\0$ 
and  $A$  describes the balance between the rate of mitosis (cell proliferation) and apoptosis (naturally cell death).
The function $f\in C^\infty([0,\infty))$ has the following properties
\begin{equation}\label{eq:conditions}
f(0)=0 \qquad\text{and}\qquad f'(\psi)>0 \quad\text{for}\quad \psi\geq0.
 \end{equation}
We already know from \cite[Theorem 1.1]{EM} that 
\begin{thm}\label{T:2}
Given $(A,G) \in(0, f(1))$, there exists
a unique radially symmetric solution $D(0,R_A)$ to problem \eqref{eq:problem}.
 The radius $R_A$ of the stationary tumor depends only on the parameter A and decreases
with respect to this variable.
 \end{thm}
 Hence $D(0,R_A)$ is a  solution of \eqref{eq:problem} for all $G\in\R.$
Thus, we may use $G$ as a bifurcation parameter to obtain also other solutions of \eqref{eq:problem}.
 This is in accordance with the numerical simulation \cite{VCris}.

In order to determine  steady-states of \eqref{eq:problem} we introduce a parametrisation for the unknown tumor domain $\0.$ 
 Therefore we define the small H\"older spaces  $h^{r}(\s)$, $r\geq 0$, 
as closure of the smooth functions $C^\infty(\s)$
in the H\"older space $C^{r}(\s)$,
whereby, $\s$ stands for the unit circle and we identify functions on $\s$ with $2\pi$-periodic functions on $\R.$
Furthermore, we fix $\alpha\in(0,1)$ and  use functions belonging to the open neighbourhood 
\begin{equation*}
\V:=\{ \rho\in h^{4+\alpha}(\s)\,:\, \|\rho\|_{C(\s)}<1/4 \},
\end{equation*}
of the zero function in $ h^{4+\alpha}(\s)$ 
to parametrise domains close to the discus $D(0,R_A)$.
Given $\rho\in\V,$ we define  domain
\[
\0_\rho:=\left\{x\in\R^2\, :\, |x|<R_A\left(1+\rho\left(x/|x|\right)\right)\right\}\cup\{0\},
\]
with boundary $\p\0_\rho=\Gamma_\rho:=\left\{R\left(1+\rho(x)\right)x\,:\,x\in\s\right\}.$
Given $x\in\Gamma_\rho,$ the real number $\rho(x/|x|)$ is the ratio of the signed distance from $x$ to the circle $R_A\cdot \s$ and $R_A$.
If $\0=\0_\h$ for some $\h\in\V$, problem \eqref{eq:problem} re-writes 
With this notation \eqref{eq:problem} is equivalent  to the following system of equations
\begin{equation}\label{3}
\left \{
\begin{array}{rlcllll}
\Delta \psi &=& f(\psi )  &\text{in} &\Omega _{\rho},  \\[1ex]
\Delta p &=& 0 & \text{in}& \Omega _{\rho},  \\[1ex]
\psi &=& 1 & \text{on}& \Gamma _{\rho},  \\[1ex]
p&=& \kappa_{\Gamma_{\rho}}- AG \displaystyle\frac{ |x|^2}{4} &\text{on}&\Gamma _{\rho}, \\[1ex]
\left<G\nabla \psi -\nabla p- AG\displaystyle \frac{ x}{2}, \nabla N_{\rho}\right> &=&0 & \text{on}& \Gamma _{\rho}, 
\end{array}
\right.
\end{equation}
where $N_\h:A(3R_A/4,5R_A/4)\to\R$ is the function  defined by $N_\rho(x):=|x|-R_A-R_A\rho(x/|x|)$ for all
$x$ in  the annulus 
\[
A(3R_A/4,5R_A/4):=\{x\in\R^2\,:\, 3R_A/4<|x|<5R_A/4\}.
\]

We re-expressed now problem \eqref{3}  as an  abstract operator equation on unit circle $\s$.
To this scope we introduce for each $\h\in\V$ the Hanzawa diffeomorphism $\Theta_\rho:\R^2\to\R^2$ by
\[
\Theta_\rho(x)=Rx+\frac{Rx}{|x|}\f(|x|-1)\rho\left(\displaystyle\frac{x}{|x|}\right),
\]
where the cut-off function $\f\in C^\infty(\R,[0,1])$ satisfies
\[
\f(r)=\left\{
\begin{array}{llll}
&1,& |r|\leq 1/4,\\[2ex]
&0,& |r|\geq 3/4,
\end{array}
\right.
\]
and additionally $\max|\f'(r)|<4.$ 
In can be easily seen that $\Theta_\rho $ is a diffeomorphism mapping $\0:=D(0,1)$ onto $ \0_\h$,
 i.e. $\Theta_\rho\in  \mbox{\it{Diff}}\,^{4+\alpha}(\0,\0_\rho)\cap \mbox{\it{Diff}}\,^{4+\alpha}(\R^2,\R^2).$
As we did in \cite{EM1} we define for each $\h\in\V$ the function $\x(\h):=\psi\circ\Theta_\h$,
whereby $\psi$ is the solution of the semilinear Dirichlet problem
\begin{equation}\label{7}
\left \{
\begin{array}{rlcllll}
\Delta \psi &=& f(\psi )  &\text{in} &\Omega _{\rho},  \\[1ex]
\psi &=& 1 & \text{on}& \Gamma _{\rho}, 
\end{array}
\right.
\end{equation}
respectively for $\h\in\V$ and $G\in\R$ we set $\mS(G,\h):=p\circ\Theta_\h$, where  $p$ solves
\begin{equation}\label{8}
\left \{
\begin{array}{rlcllll}
\Delta p &=& 0 & \text{in}& \Omega _{\rho},  \\[1ex]
p&=& \kappa_{\Gamma_{\rho}}- AG \displaystyle\frac{ |x|^2}{4} &\text{on}&\Gamma _{\rho}.
\end{array}
\right.
\end{equation}
With this notation, our problem \eqref{3} reduces to the operator equation
\begin{equation}\label{9}
\text{$\Phi(G,\h)=0 $ in $h^{1+\alpha}(\s)$}
\end{equation}
where $\Phi:\R\times \V\to h^{1+\alpha}(\s)$ is the nonlinear and nonlocal operator defined by
\begin{equation}\label{D:P}
\Phi(G,\h):=\left<G\nabla \left(\x(\h)\circ\Theta^{-1}_\h\right) -\nabla \left(\mS(G,\h)\circ\Theta^{-1}_\h\right)- AG\displaystyle \frac{ x}{2}, \nabla N_{\rho}\right>\circ\Theta_\h.
\end{equation}
The function $\Phi$ is smooth $\Phi\in C^\infty(\R\times\V, h^{1+\alpha}(\s))$, cf. \cite{EM1}, and  the steady-state $D(0,R_A)$
corresponds to the function  $\h=0$ which is solution of \eqref{9} for all $G\in\R.$

Therefore, we shall refer to $\Sigma:=\{(G,0)\,:\, G\in\R\} $ as the set of trivial solutions of \eqref{9}.
The main result of this paper, Theorem \ref{BifTh} states that 
 there exist infinitely many local bifurcation branches emerging from $\Sigma$ and  which consist only of  solutions of the original problem \eqref{eq:problem}.
 Our analysis is based on the fact that we could determine an explicit formula for the partial derivative $\p_\h\Phi(G,0)$ cf. \cite{EM1}.
 Given $\h\in h^{4+\alpha}(\s),$  we let $\h=\sum_{k\in\Z}\wh\h(k)x^k$ denote its associated Fourier series. 
Then we have: 
\begin{equation}\label{eq:PHI}
\p\Phi(G,0)\left[\sum_{k\in\Z}\wh\h(k)x^k\right] 
=\underset{k\in \Z}\sum \mu_k(G) \widehat \h(k) x^k, 
\end{equation}
where the symbol $(\mu_k(G))_{k\in\Z}$ is given by the relation
\begin{equation}\label{eq:symbol}
\mu_k(G):= -\frac{1}{R_A^3}|k|^3+ \frac{1}{R_A^3}|k| - G \left(\frac{A}{2}\frac{u_{|k|}'(1)}{u_{|k|}(1)}+A-f(1)\right),
\end{equation}
and $u_{|k|}\in C^\infty([0,1])$ is the solution of  the initial value problem
\begin{equation}\label{uniculu}
\left\{
\begin{array}{rlll}
u''+\displaystyle\frac{2n+1}{r}u' &=& R_A^2f'(v_0)u, \quad & 0<r<1,\\[2ex]
u(0) &=& 1, \\[2ex]
u'(0)&=&0,
\end{array}
\right.
\end{equation}
 when $n=|k|$ and $v_0:=\x(0).$ 
By the weak maximum principle $v_0$ is radially symmetric, so that we identified in \eqref{uniculu} $v_0$ with its restriction to the segment $[0,1].$ 
Particularly, \eqref{eq:PHI} and \eqref{eq:symbol} imply that
\begin{equation}\label{eq:spectr}
\sigma(\p\Phi(G,0))=\{\mu_k(G)\,:\, k\in\Z\}.
\end{equation}
   
Before stating our main result we study first the properties of the symbol $(\mu_k(G))_{k\in\Z}.$
\begin{lem}\label{L:pre}
There exists $M>0$ such that
\begin{equation}\label{eq:est}
\text{$u_k'(1)\leq \frac{M}{2k+2}$ and  $u_k(1)\leq1+\frac{M}{(2k+1)(2k+3)}$}
\end{equation}
for all $k\in\N.$
\end{lem}   
  \begin{proof}
  Let $k\in \N$ be fixed. 
Since
\begin{align}\label{50}
u_k(r)=1+ \int_0^r \frac{R_A^2}{s^{2k+1}}\int_0^s \tau^{2k+1}f'(v_0(\tau))u_k(\tau)\, d\tau \, ds, \quad 0\leq r\leq 1,
\end{align}
we deduce that $u_k$ is strictly increasing for all $k\in\N.$
Let now $v:=u_{k+1}-u_k$.
From  \eqref{uniculu} we obtain
\begin{equation*}
\left \{
\begin{array}{rlll}
v''+\displaystyle\frac{2k+1}{r}v'&=& R^2 f'(v_0)v-\displaystyle\frac{2}{r}u_{k+1}', &0< r< 1,\\[1ex]
v'(0)&=& 0,\\[1ex]
v(0)&=& 0.
\end{array}
\right.
\end{equation*}
Furthermore,  we have
\begin{align*}
\underset{t\to 0}\lim \frac{v'(t)}{t} &= \underset{t\to 0}\lim \frac{u_{k+1}'(t)-u_k'(t)}{t}= R_A^2f'(v_0(0))\left( \frac{1}{2k+4}-\frac{1}{2k+2}\right)\\[1ex]
&= -R_A^2 f'(v_0(0)) \frac{2}{(2k+4)(2k+2)} < 0,
\end{align*}
which implies by \eqref{eq:conditions} that $v'(t) <0$ for $t\in (0, \delta)$ and  some  $\delta<1.$ 
Thus, $v$ is decreasing on $(0, \delta).$
Let now  $t\in [0,1]$ and set $m_t:= \min_{[0,t]} v \leq 0$.
A maximum principle argument shows   that  the nonpositive minimum must be achieved at $t$,  $m_t=v(t)$ which implies $u_{k+1}(t)\leq u_k(t)$ for all $t\in [0,1]$.
Particularly, $u_{k+1}(1) \leq u_k(1)$.

We prove now the estimate for $u_k'(1).$
Setting $M:= R_A^2 u_0(1) \max_{[0,1]} f'(v_0)$
we obtain, due to \eqref{50}, that 
\begin{align*}
u_k'(1)&= \int_0^1 R_A^2 \int_0^s \tau^{2k+1}f'(v_0(\tau))u_k(\tau) \, d\tau \\[1ex]
&\leq M \int_0^1 \, \int_0^s \tau^{2k+1} \, d\tau= \frac{M}{(2k+1)(2k+3)}.
\end{align*}
The estimate for $u_k$ follows similarly. 
\end{proof} 
   
By Lemma \ref{L:pre} we know that $u_k'(1)/u_k(1)\to_{k\to\infty}0.$
Therefore, we may define for $k\in\N$ with    
\[
\frac{A}{2} \frac{u_k'(1)}{u_k(1)}+ A-f(1) \neq 0,
\]
the constant
\begin{align}\label{Gk}
G_k:=\displaystyle\frac{\displaystyle{-\frac{1}{R^3}k^3+\frac{1}{R^3}k}}{\displaystyle{\frac{A}{2} \frac{u_k'(1)}{u_k(1)}+ A-f(1)}},
\end{align} 
which is the only natural number such that $\mu_k(G_k)=0.$
We do this since we cannot estimate whether $\mu_k(G)$, with $k$ small are zero or not.
For example, we know from \cite{EM1} that $\mu_1(G)=0$ for all $G\in\R,$ which makes the things difficult when trying to apply  bifurcation    theorems to \eqref{9}. 
In virtue of Lemma \ref{L:pre} we have:
\begin{lem}\label{propGk}
There exists $k_1\in\N$  with the property that $0<G_k <G_{k+1}$ for all $k\geq k_1.$ 
\end{lem}
\begin{proof}
From \eqref{50} we obtain  by partial integration that
\begin{align*}
k\cdot (u'_{k}(1)-u'_{k+1}(1)) =& \,R_A^2 \int_0^1 k \tau^{2k+1} f'(v_0(\tau)) [u_k(\tau)-\tau^2 u_{k+1}( \tau)] \, d\tau \\[1ex]
= &\, R_A^2 k\left.\frac{\tau^{2k+2}}{2k+2}f'(v_0(\tau)) [u_k(\tau)-\tau^2 u_{k+1}( \tau)] \right|_0^1\\[1ex]
&- \int_0^1 k\frac{\tau^{2k+2}}{2k+2}\left\{  f''(v_0(\tau)) \left [u_k(\tau)-\tau^2 u_{k+1}( \tau)\right] \phantom{\int}\right.\\[1ex]
&+ \left. f'(v_0(\tau))\phantom{\int} \hspace{-0.2cm}[u_k'(\tau)-2\tau u_{k+1}(\tau) -\tau^2 u_{k+1}'(\tau)]  \right\}\, d\tau \\[1ex]
\leq & \,R_A^2f'(1) \frac{k}{2k+2} (u_k(1)- u_{k+1}( 1))+ L\frac{k}{(k+1)^2},
\end{align*}
with a constant $L$ independent of $k$. 
Letting now $k \to \infty$ we get 
\begin{align}\label{anan+1}
k\cdot (u'_{k}(1)-u'_{k+1}(1))\underset {k \to \infty }\longrightarrow 0.
\end{align}
Having this estimates at hand we can prove now the assertion of the lemma. 
For simplicity we set  $C:= 2(f(1)-A)/A >0$, $a_k:=u_k'(1),$ and $b_k:=1+u_k(1).$ 
In virtue of
$$G_k = \frac{2}{AR^3} \frac{\displaystyle{k^3-k}}{\displaystyle{C- \frac{u_k'(1)}{u_k(1)}}}. $$
 we compute
 \begin{align*}
G_{k+1}> G_k \Leftrightarrow \, &  \frac{(k+1)^3-(k+1)}{C-  \displaystyle\frac{a_{k+1}}{1+b_{k+1}}} >  \frac{k^3-k}{C- \displaystyle\frac{a_k}{1+b_k}}\\[1ex]
\Leftrightarrow \,&  C(k+1)^3-C(k+1)-Ck^3+Ck \\[1ex]
&> \frac{a_{k+1}}{1+b_{k+1}}(k-k^3)+ \frac{a_k}{1+b_k} [(k+1)^3-(k+1)]\\[1ex]
\Leftrightarrow \, &  C(3k^2+3k) > \frac{a_{k+1}(k-k^3)+a_k(k^3+3k^2+2k)}{(1+b_k)(1+b_{k+1})}\\[1ex]
&\phantom{aaaaaaaaaa} +\frac{a_{k+1}b_k (k-k^3)+a_k b_{k+1}(k^3+3k^2+2k)}{(1+b_k)(1+b_{k+1})}\\[1ex]
\Leftrightarrow \, &  C(3k^2+3k) > \frac{k^3(a_k-a_{k+1})+a_{k+1}k+a_k(3k^2+2k)}{(1+b_k)(1+b_{k+1})}\\[1ex]
&\phantom{aaaaaaaaaa} +\frac{a_{k+1}b_k (k-k^3)+a_k b_{k+1}(k^3+3k^2+2k)}{(1+b_k)(1+b_{k+1})}.
\end{align*}
 Taking into consideration the relations \eqref{eq:est} and \eqref{anan+1}, we find a positive integer $k_1$
such that  strict inequality holds in the last relation above for all $k\geq k_1.$
\end{proof}

Finally, we set  
\[
 G_\bullet:= \frac{\displaystyle\frac{1}{R_A^3 }k_1^3 
-\frac{1}{R_A^3 }k_1}{\underset{0\leq k\leq k_1}\min  \left \{\left |f(1)-A-\displaystyle\frac{A}{2} 
\displaystyle\frac{u_k'(1)}{u_k(1)}\right|:f(1)-A-\displaystyle\frac{A}{2} \frac{u_k'(1)}{u_k(1)}\not =0\right\}}.
\]

 The main result of this paper is the following theorem:
\begin{thm}\label{BifTh}
Assume  $A\in(0,f(1))$ and that
\begin{equation}\label{eq:feri}
\frac{A}{2} \frac{u_0'(1)}{u_0(1)}+A-f(1) \not = 0 
\end{equation} 
holds true.
Let $l\geq 2$ be fixed and  $k\in \N,$ $k\geq 1$ such that $G_{kl}>G_\bullet.$ 
The pair $(G_{kl}, 0)$ is a bifurcation point from the trivial solution $\Sigma$.
More precisely, in a suitable neighbourhood of $(G_{kl}, 0),$ there exists 
a smooth branch of solutions $\left(G^{kl}(\varepsilon), \rho^{kl}(\varepsilon )\right)$ of problem \eqref{9}.
 For $\e\to0$, we have the following asymptotic expressions:
 \begin{eqnarray*}
 &&G^{kl}(\varepsilon )= G_{kl}+ O(\e), \\[1ex]
 &&\rho^{kl}(\varepsilon )=\e\cos(kls)+O(\e^2).
\end{eqnarray*}
Moreover, any other $G>  G_\bullet,$ $G\not\in \{G_{kl}:k\geq1\},$  is not a bifurcation point.
\end{thm}

Condition \eqref{eq:feri} is not to restrictive.
Particularly, if $R_A=1$  and $f=\id_{[0,\infty)}$ we have shown in \cite{EM1} 
that \eqref{eq:feri} is satisfied.

\section{Proof of the main result}
The main tool we use when proving Theorem \ref{BifTh} is the classical bifurcation result on bifurcations from simple eigenvalues due to Crandall and Rabinowitz:
\begin{thm}[see \cite{BT, CR}] \label{ThmCR}Let $X, $ $Y$ be real Banach spaces and $G(\lambda,u)$ be a $C^q$ $(q\geq3)$
mapping from a neighbourhood of a point $(\lambda_0,u_0)\in \R\times X$ into $Y$.
Let the following assumptions hold:

\begin{itemize}
\item[(i)] $G(\lambda_0,u_0)=0, \, \partial_\lambda G(\lambda_0,u_0)=0,$
\item[(ii)] $\Ker \p_uG(\lambda_0,u_0)$ is one dimensional, spanned by $v_0,$
\item[(iii)] $\im \p_uG(\lambda_0,u_0)$ has codimension 1,
\item[(iv)] $ \p_\lambda\p_\lambda G(\lambda_0,u_0)\in \im \p_uG(\lambda_0,u_0)$, 
$\p_\lambda\p_u G(\lambda_0,u_0)v_0\notin \im \p_uG(\lambda_0,u_0).$
\end{itemize}
Then $(\lambda_0,u_0)$ is a bifurcation point of the equation 
\begin{equation}\label{CR}
G(\lambda,u)=0
\end{equation}
 in the following sense: In a neighbourhood of $(\lambda_0,u_0)$ the set of solutions of equation \eqref{CR}
consists of two $C^{q-2}$ curves $\Sigma_1$ and $\Sigma_2$, which intersect only at the point $(\lambda_0,u_0).$ 
Furthermore, $\Sigma_1$, $\Sigma_2$ can be parameterised as follows:
\begin{itemize}
\item[$\Sigma_1:$] $(\lambda, u(\lambda)),$ $|\lambda-\lambda_0| \, \mbox{is} \, small, \, u(\lambda_0)=u_0, \, u'(\lambda_0)=0,$
\item[$\Sigma_2:$] $(\lambda(\e),u(\e)),$ $|\e| \, \mbox{is} \, small, \,  (\lambda(0),u(0))=(\lambda_0,u_0), \, u'(0)=v_0.$
\end{itemize}
\end{thm}
We want to apply Theorem \ref{ThmCR} to the particular problem \eqref{9}.
As we already mentioned
$\Phi(G,0)=0$ for all $G\in\R.$ 
However, since we can not estimate the eigenvalues $\mu_k(G), $ $1\leq k\leq k_1$ we have to eliminate them from the spectrum of $\p\Phi(G,0),$
 and also to reduce the dimensions of the eigenspace 
corresponding to an eigenvalue  $\mu_k(G)$, $k\geq k_1,$ which we may chose to be equal to $0$ if $G$ is large enough,  to one.
This is due to the fact that the dimension of the eigenspace corresponding to an arbitrary eigenvalue $\mu_k(G)$ is larger then $2,$ since $x^k $ and $x^{-k}$ are eigenvectors 
of this eigenvalue. 

This may be done by restricting the operator $\Phi$ to spaces consisting only of
 $2\pi/l-$periodic and even functions.
Given $k\in\N$ and $l\in\N, l\geq2$, we define 
 \[
h^{k+\alpha}_{e,l}(\s):=\{ \text{$\h\in h^{k+\alpha}(\s)\,:\, \h(x)=\h(\ov x)$ and $ \h(x)=\h(e^{2\pi i/l}x)$ for all $x\in\s$}\},
\]
where for $x\in \C,$ $\ov x$ denotes its  complex conjugate.
Set further $\V_{e,l}:=\V\cap h^{4+\alpha}_{e,l}(\s).$
By identifying functions on $\s$ with $2\pi-$periodic functions on $\R$ we can expand $\h\in h^{k+\alpha}_{e,l}(\s)$
in the following way
\[
\h(s)=\sum_{k=0}^\infty a_k \cos(kls),
\]
where
$a_k=2\wh\h(kl)$ for $k\geq0.$
With this notation we have:
\begin{lem}\label{L:rest2}
Given $l\geq 2$, the operator $\Phi$ maps smoothly $\R\times \V_{e,l}$ into $h^{1+\alpha}_{e,l}(\s),$
i.e. 
\[
\phi\in C^\infty(\R\times \V_{e,l}, h^{1+\alpha}_{e,l}(\s)).
\]
\end{lem}
\begin{proof}
The proof is similar to that of \cite[Lemma 5.5.2]{AM} and therefore we omit it.
\end{proof}

Finally, we come  to the proof of our main result.
\begin{proof}[Proof of Theorem \ref{BifTh}]
Fix now $l\geq 2$. 
We infer from
relation \eqref{eq:PHI}  that the partial  derivative
of the smooth mapping $\Phi:\R\times \V_{e,l}\to h^{1+\alpha}_{e,l}(\s)$ with respect to $\h$ at $(G,0)$ is the Fourier multiplier
\begin{equation}\label{drhophi}
\p_\h \Phi(G, 0)\left[ \sum_{k=0}^\infty a_k \cos(kls)\right] = \sum_{k=0}^\infty \mu_{kl} (G)a_k \cos(kls)
\end{equation}
for all $\h= \sum_{k=0}^\infty a_k \cos(klx)\in h^{4+\alpha}_{e,l}(\s),$ 
where $\mu_{kl}(G)$, $k\in\N,$ is defined by \eqref{eq:symbol}.
Our  assumption \eqref{eq:feri} implies  that $\mu_0(G) \not = 0.$ 

The proof is based on the following observation: if $G>G_\bullet$ and $\mu_{kl}(G)=0$ then it must hold that $G=G_{kl} $ and $kl> k_1.$
Indeed, we  notice  that if  $\mu_{kl}(G)=0,$ for some $kl\leq k_1$,  
then 
\begin{align*}
  G_\bullet < G=
\frac{\displaystyle \frac{1}{R_A^3 }(kl)^3 -\displaystyle\frac{1}{R_A^3 }kl}
{ \left |f(1)-A-\displaystyle\frac{A}{2}\displaystyle \frac{u_{kl}'(1)}{u_{kl}(1)}\right|}
\leq  G_\bullet,
\end{align*}
since, by \eqref{eq:feri} and $l\geq2$, $kl\geq2.$
If $kl> k_1$, then $\mu_{kl}(G)=0 $ iff
\begin{align*}
 G =\frac{-\displaystyle\frac{1}{R_A^3 }(kl)^3 +\displaystyle\frac{1}{R_A^3 }kl}
{\displaystyle\frac{A}{2} \displaystyle \frac{u_{kl}'(1)}{u_{kl}(1)}+A-f(1)}=G_{kl}.
\end{align*}

Let $1\leq k\in \N$ be given such that $G_{kl}>G_\bullet$.
From Lemma \ref{propGk} and the previous observation we get $\mu_{ml}(G_{kl})\neq 0$ for $m\neq k.$ 
Our assumption \eqref{eq:feri} ensures that the Fr\' echet derivative $\p_\h \Phi(G_{ml}, 0)$ of the restriction 
$\Phi:\R\times \V_{e,l} \subset h^{4+\alpha}_{e,l}(\s)\to h^{1+\alpha}_{e,l}(\s),$ which is given
 by relation \eqref{drhophi}, has a one dimensional kernel spanned by $\cos(kls).$ 
 We also  observe that its image is closed and has codimension  equal to one.

We are left now to prove  that the transversality condition $(iv)$ of Theorem \ref{ThmCR} holds. 
Since $G_{kl}>G_\bullet$, our observation implies  $kl\geq k_1+1$ and so  
\[
-\left( \frac{A}{2}\frac{u_{kl}'(1)}{u_{kl}(1)}+ A-f(1)\right)>0.
\]
Further on,  we get from relation \eqref{drhophi} and \eqref{eq:symbol} that
\begin{equation*}
\p_G\p_\h\Phi(0)\left[\cos(kls)\right]=- \left( \frac{A}{2}\frac{u_{kl}'(1)}{u_{kl}(1)}+ A-f(1)\right) \cos(kls),
\end{equation*}
and since $\cos(kls)\not \in \im \p_\h \Phi(G_{kl}, 0)$ we deduce that the assumption of  Theorem \ref{ThmCR} are all verified.
By applying Theorem \ref{ThmCR} we obtain the bifurcation result stated in Theorem \ref{BifTh}
 and the asymptotic expressions  for the bifurcation branches $\left(G^{kl}(\varepsilon), \rho^{kl}(\varepsilon )\right).$ 

Moreover, if $G>  G_\bullet$ and $G\neq G_{kl}$ for all $k$ with $kl\geq k_1+1$, then it must hold that
$\mu_{kl}(G)\neq 0$ for all $k\in \N,$ and we may apply \cite[Theorem 4.5.1]{AM} to obtain that $\p_\h \Phi(G, 0)$ is an isomorphisms.
The Implicit function theorem then states that $(G, 0)$ is not a bifurcation point and the proof is completed.
\end{proof}
The possible steady-states of problem \eqref{eq:problem} are depicted in Figure $1$.
\begin{figure}\label{F:bifu} 
$$\includegraphics[width=0.9\linewidth]{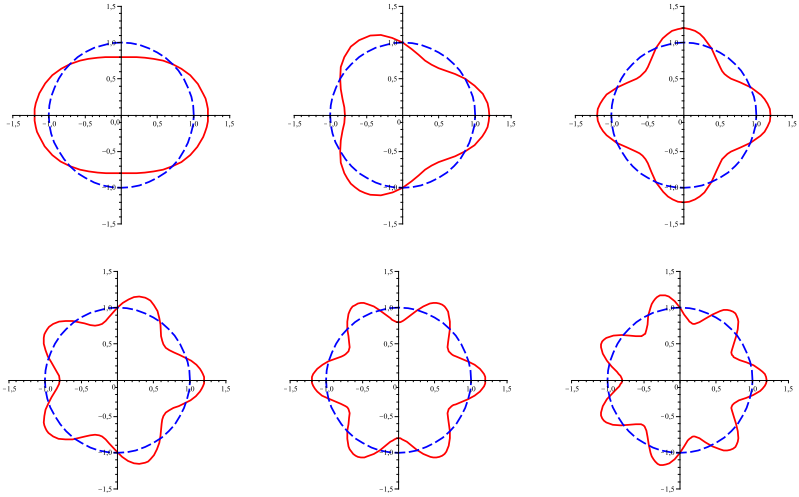}$$
\caption{Possible steady-states of the problem \eqref{eq:problem} }
 \end{figure}
\vspace{1cm}

\end{document}